\documentclass[A4,11pt]{article}
\usepackage{geometry}
\usepackage{amsthm,amsmath,amssymb,amsfonts,hyperref}
\usepackage{microtype,mathrsfs}
\usepackage{epsfig,color}
\usepackage{mathrsfs}
\usepackage{dsfont}
\usepackage{enumerate}
\usepackage{enumitem}
\setlist[enumerate]{leftmargin=*}

\newtheorem{thm}{Theorem}[section]
\newtheorem{prop}{Proposition}[section]
\newtheorem{cor}{Corollary}[section]
\newtheorem{rmk}{Remark}[section]
\newtheorem{lma}{Lemma}[section]

\newcommand{\was}{d_W}
\newcommand{\mO}{\mathcal O}

\newcommand{\eq}{\begin{equation}}
\newcommand{\qe}{\end{equation}}
\newcommand{\sqn}{\sqrt n}
\newcommand{\tth}{\tilde \theta}
\newcommand{\tf}{\tilde f}

\def\N{{\rm I\kern-0.16em N}}
\def\R{{\rm I\kern-0.16em R}}
\def\E{{\rm I\kern-0.16em E}}
\def\P{{\rm I\kern-0.16em P}}
\def\F{{\rm I\kern-0.16em F}}
\def\B{{\rm I\kern-0.16em B}}
\def\C{{\rm I\kern-0.46em C}}
\def\G{{\rm I\kern-0.50em G}}

\newcommand{\Ecro}[2][]{\E_{#1}\left[#2\right]}
\newcommand{\Pcro}[2][]{\P_{#1}\left[#2\right]}

\newcommand{\lip}{\text{Lip}}

\numberwithin{equation}{section}
\font\eka=cmex10
\usepackage{ae}
\def\ind{\mathrel{\hbox{\rlap{%
\hbox to 7.5pt{\hrulefill}}\raise6.6pt\hbox{\eka\char'167}}}}
\begin{document}

\title{\textbf{On the rate of convergence in de Finetti's representation theorem}}

\date{}
\renewcommand{\thefootnote}{\fnsymbol{footnote}}
\begin{small}
\author{Guillaume Mijoule\footnotemark[1],\, Giovanni Peccati\footnotemark[2] \, and   Yvik Swan\footnotemark[1]}
\end{small}
\footnotetext[1]{Université de Liège.}
\footnotetext[2]{Université du Luxembourg. }

\maketitle

\abstract {}{A consequence of de Finetti's representation theorem is that for every infinite sequence of exchangeable 0-1 random variables $(X_k)_{k\geq1}$, there exists a probability measure $\mu$ on the Borel sets of $[0,1]$ such that $\bar X_n = n^{-1} \sum_{i=1}^n X_i$ converges weakly to $\mu$. For a wide class of probability measures $\mu$ having smooth density on $(0,1)$, we give bounds of order $1/n$ with explicit constants for the Wasserstein distance between the law of $\bar X_n$ and $\mu$. This extends a recent result {by} Goldstein and Reinert \cite{goldstein2013stein} regarding the distance between the scaled number of white balls drawn in a P\'olya-Eggenberger urn and its limiting distribution.  We prove also that, in the most general cases, the distance between the law of $\bar X_n$ and $\mu$ is bounded below by $1/n$ and above by $1/\sqn$ (up to some multiplicative constants). For every $\delta \in [1/2,1]$, we give an example of an exchangeable sequence such that this distance is of order $1/n^\delta$.


\medskip

\noindent
{\it Keywords:} de Finetti's theorem, Exchangeable Variables, Wasserstein distance, Urn models.

\smallskip

\noindent
{\it 2010 AMS subject classification:}  60F05, 60G09.

\tableofcontents

\section{Introduction}

\subsection{Overview and framework}

An infinite sequence $(X_k)_{k\geq 1}$ of random variables is exchangeable if for every $n \geq 1$ and every permutation $\sigma$ of $\{1,\ldots,n\}$, $(X_{\sigma(1)},\ldots,X_{\sigma(n)})$ has the same distribution as $(X_1,\ldots,X_n)$. The following {fundamental} theorem was discovered by Bruno de Finetti \cite{de1980foresight} :

\begin{thm}(de Finetti, 1937)
\label{thm:definetti}
An infinite sequence $(X_k)_{k\geq 1}$   of 0-1 random variables is exchangeable if only if there exists a (necessarily unique) probability measure $\mu$ on the Borel sets of $[0,1]$ such that for every $n\geq 1$ and every $(e_1,\ldots,e_n) \in \{0,1\}^n$,
\begin{equation}
\label{eq:definetti}
\P\left[ X_1 = e_1, \ldots,X_n=e_n\right] = \int_0^{1} t^{k} (1-t) ^{n-k}\, \mu(dt),
\end{equation}
where $k = \sum_{i=1}^n e_i$.
\end{thm}

Exchangeability has been extensively studied in the literature. Hewitt and Savage \cite{hewitt1955symmetric} extend de Finetti's result for variables taking values in general spaces. Diaconis and Freedman \cite{diaconis1980finite} give an approximation result when the sequence $(X_k)_{k\geq 1}$ is finite, in which case, a representation of the type \eqref{eq:definetti} does not necessarily hold.  
For an overview of results related to exchangeability, we refer to the classical {lecture notes} \cite{aldous1985exchangeability}, {as well as to \cite{austin, kallenberg2006probabilistic, pitman} for more recent accounts.}

Equation $\eqref{eq:definetti}$ has an elegant Bayesian interpretation, namely: the law of $(X_k)_{k\geq1}$ is that of a sequence of i.i.d.\ Bernoulli random variables with parameter $\theta$ randomly chosen from the (prior) probability measure $\mu$.  The measure $\mu$ is sometimes called {\it de Finetti's measure} or {\it mixing measure} {associated with the sequence}.

Defining $\bar X_n = n^{-1} \sum_{i=1}^n X_i$, we readily obtain {the following} De Finetti-type Law of Large Numbers (LLN) {in distribution}:
\begin{equation}
  \label{eq:2}
  \bar X_n \stackrel{\mathcal{L}}{\longrightarrow} \mu,
\end{equation}
where $\mathcal{L}$ indicates weak convergence. Relation \eqref{eq:2} is easy to see with the Bayesian point of view of \eqref{eq:definetti}~: if, on some probability space, we are given a random variable $\theta$ with distribution $\mu$ and a sequence $(X_k)_{k\geq 1}$ which are, conditionally on $\theta$, Bernoulli i.i.d. random variables with parameter $\theta$, then $\bar X_n$ converges almost surely to $\theta$. Hence, convergence in distribution also holds.

Conditions under which LLNs for exchangeable sequences hold have, naturally, been extensively studied in the literature see \cite{MR0185643,MR2204712,stoica2011complete,MR901022} or the more general \cite{MR860208}.  There has however been only little investigation into explicit rates of convergence {for the distributional limit theorem} in \eqref{eq:2} for general mixture measures $\mu$. One of the results of \cite{hauray2014kac} is a bound of the order $1/\sqn$ whenever the $X_k$ take values in a subspace of $\R^d$; however, in our much simpler framework where $X_k$ is 0 or 1, such a bound is not hard to obtain directly (see Proposition \ref{prop:boundbelow}). We also mention that, although the main result of \cite{diaconis1980finite} is sometimes refered to as {\it quantitative de Finetti theorem} (or {\it quantum de Finetti theorem}), it is a completely different problem we investigate here : in \cite{diaconis1980finite}, the exchangeable sequence is supposed to be finite, and a bound is obtained for the distance between the {distribution} of $(X_1,\ldots,X_n)$ and the set of mixture measures of the type \eqref{eq:definetti}. Of course, if the sequence is infinite (which is what we assume), de Finetti's theorem exactly says that this distance is zero.

To the best of our knowledge, the closest result to ours is due to \cite{goldstein2013stein} and concerns the  classical {\it P\'olya-Eggenberger urn model} {(see e.g. \cite{pitman} for a general discussion, as well as \cite{edp,edpp,p} for several recent developments)}. This model is constructed as follows : at time 0, an urn contains $A\geq 1$ white balls and $B \geq 1$ black balls and at every positive integer time, a ball is randomly drawn from the urn (independently of the past) and replaced along with $m\geq 1$ additional balls of the same color. Then, defining $X_n=1$ if a white ball is drawn at time $n$ and $X_n=0$ otherwise, it is well known that \eqref{eq:definetti} holds with $\mu$ being the Beta distribution with parameters $A/m$ and $B/m$. Here, the Beta distribution with parameters $\alpha>0$ and $\beta>0$ is the probability measure with density
\eq
\label{eq:densitebeta}
p_{\alpha,\beta} (x) = \frac{1}{B(\alpha,\beta)}x^{\alpha-1}(1-x)^{\beta-1} \mathbf 1_{(0,1)}(x), \qe where $B(\alpha,\beta) = \int_0^1 t^{\alpha-1}(1-t)^{\beta-1} dt$ is the Beta function and $\mathbf 1_E$ the indicator function of the set $E$. One of the results of \cite{goldstein2013stein} is that the Wasserstein distance between the scaled number of drawn white balls $\bar X_n$ and its corresponding limiting Beta distribution is of order $1/n$. The proof is based on a version of Stein's method as adapted to the Beta distribution : a Stein operator is found for the discrete variable $\bar X_n$ and compared to the Stein operator of the Beta distribution. We also mention a similar bound obtained by Döbler \cite{dobler2014stein}.


{We recall that the Wasserstein distance between the distributions of two real-valued and integrable random variables $X,Y$ is given by the quantity 
\eq
\label{eq:was3}
\was(X,Y) = \sup_{\phi \in \lip(1)}\left | \E[ \phi(X)] - \E[\phi(Y)]\right |,
\qe
where $\lip(K)$ is the set of $K$-Lipschitz functions on the real-line. It is a well-known (and easily checked) fact that the topology induced by $d_W$ on the class of probability measures on the real line is strictly stronger than the topology of convergence in distribution. In the framework of the present paper, it is also interesting to notice that, if one restricts oneself to the collection of all probability measures supported on $[0,1]$ then the two topologies are actually equivalent (to see this one can e.g. use the representation \eqref{eq:was4} below, and then exploit the Dominated Convergence theorem). } 

It is the goal of this paper to estimate the rate of convergence in  {Wasserstein distance} {for the distributional limit theorem} in \eqref{eq:2}.
\subsection{Main results}

In this paper, we prove that a bound of the same order as in \cite{goldstein2013stein} for the Beta target still holds for more general distributions $\mu$. Our main theorem is the following : \begin{thm} \label{thm:main} Let $(X_k)_{k\geq 1}$ be an infinite sequence of 0-1 exchangeable variables, $\bar X_n = n^{-1}\sum_{i=1}^n X_i$ and $\mu$ the limiting distribution of $\bar X_n$. Suppose $\mu$ has a smooth density $p$ on $(0,1)$ satisfying \begin{equation}
  \label{eq:3}
  \int_0^1 u(1-u) |p'(u)| du < \infty
\end{equation}
 Then with $\theta \sim \mu$,
\eq
\label{eq:ineqmain}
\frac{C_1}{n}\leq \was(\bar X_n, \theta) \leq \frac{C_2}{n},
\qe
where $C_1$ and $C_2$ only depend on $p$ and are given by
\begin{align}
 \label{eq:constantsmain}
   C_1(\mu)& = \int_0^1 u(1-u) p(u) du,\\
C_2(\mu) &= \int_0^1 \left( |1-2u| + u^2 + (1-u)^2 \right) p(u) du + \int_0^1 u(1-u) |p'(u)| du   + \frac{3}{\sqrt{2\pi e}}. \nonumber
\end{align}
\end{thm}

{
\begin{rmk} In view of the previous discussion and of the explicit expression of the constants $C_1, C_2$, it is in principle possible to obtain estimates similar to \eqref{eq:ineqmain} to more general situations, like for instance to the case where the measure $\mu$ can be represented as the weak limit of measures of the form $\mu_n = p_n dx$, where each $p_n$ is a smooth density such that the numerical sequence $ n\mapsto \int_0^1 |p_n(x)|+|p'_n(x)| dx$ is bounded. We leave such an extension to the interested reader.
\end{rmk}
}

Theorem \ref{thm:main} will be applied to the case of the Beta distribution, leading to a bound in Wasserstein distance in the P\'olya-Eggenberg urn model as explained in Section \ref{thm:definetti} (see Corollary \ref{cor:polya}). We will numerically compare the constants we obtain with the constants in \cite{goldstein2013stein}; it turns out that our result leads to better constants for the wide range of values of $A,B$ and $m$ we investigated.

{In view of the above mentioned LLN}, it is not hard to prove (see Proposition \ref{prop:boundbelow}) that, whatever the distribution $\mu$ of $\theta$, $$\frac{\Ecro{\theta(1-\theta)}}{n} \le \was(\bar X_n,\theta) \le \sqrt\frac{\Ecro{\theta(1-\theta)}}{n}.$$
Another contribution of this paper is that we prove that such bounds are sharp in the sense that, for every $\delta \in [1/2,1]$, we exhibit  a measure $\mu$ which violates Assumption   \eqref{eq:3} and such that $\was(\bar X_n,\theta)$ is of order $1/n^\delta$; see Proposition \ref{prop:examples}.

Let us now briefly sketch our strategy. We know from the classical central limit theorem that, conditionally on $\theta$, $\sqn(\bar X_n - \theta)$ converges weakly to a normal distribution with variance $\theta(1-\theta)$ (the variance of a Bernoulli variable of parameter $\theta$). Moreover, a Berry-Ess\'een type theorem gives a bound of the Wasserstein distance between those two variables. This will allow us to prove (see Proposition \ref{prop:equivalence}) that controlling $\was(\bar X_n,\theta)$ is equivalent (in a sense which will be made precise later on) to controlling $\was\left(\theta+\sqrt\frac{{\theta(1-\theta)}}{n} Z,\theta\right)$ where $Z$ stands for a standard normal random variable independent of $\theta$. Finally, we bound the latter quantity by a purely analytical method, using the representation of the Wasserstein distance as the $L^1$ norm of the difference of the cumulative distribution functions.

The paper is organized as follows. Section \ref{sec:prelim} gives the basic definitions and notations. In Section \ref{sec:boundsgeneral}, we study the trivial cases and we show that the distance between $\was(\bar X_n,\theta)$ and $\was\left(\theta+\sqrt\frac{{\theta(1-\theta)}}{n} Z,\theta\right)$ is bounded by $\mO(1/n)$. In Section \ref{sec:smooth} we bound  $\was\left(\theta+\sqrt\frac{{\theta(1-\theta)}}{n} Z,\theta\right)$ from above  for regular enough measures $\mu$ and prove Theorem \ref{thm:main}. Finally, in Section \ref{sec:examples}, for every $\delta \in [1/2,1]$, we give an example of a measure $\mu$ for which the rate of convergence is exactly of order $1/n^{\delta}$.

\section{Definitions and notations}
\label{sec:prelim}

Let $\mu$ and $\nu$ be two probability measures on $\R$. The Wasserstein (or Kantorovitch) distance between $\mu$ and $\nu$ is defined by
\eq
\label{eq:was}
\was(\mu,\nu) = \text{inf}_{\pi} \; \int\int|x-y| \pi(dx,dy),
\qe
where the infimum is taken over all probability measures $\pi$ on $\R \times \R$ with first marginal $\mu$ and second marginal $\nu$. When $X$ and $Y$ are integrable real-valued random variables, $\was(X,Y)$ will denote the Wasserstein distance between the probability measures induced by $X$ and $Y$ on the Borel sets of $\R$. In this case equation \eqref{eq:was} becomes
\eq
\label{eq:was2}
\was(X,Y) = \inf \E[|X'-Y'|],
\qe
the infimum being taken over all couples of real-valued random variables $(X',Y')$ such that $X'$ (resp. $Y'$) has the law of $X$ (resp. $Y$). From the Kantorovitch duality theorem, {we readily deduce the representation \eqref{eq:3} mentioned in the Introduction}. Yet another representation of the Wasserstein distance is given by the $L^1$-norm of the difference between the cumulative distribution functions :
\eq
\label{eq:was4}
\was(X,Y) = \int_\R \big| \P[X \leq x] - \P[Y\leq x]\big| \;dx.
\qe
{For a proof of these equivalent definitions, one can consult e.g. \cite{gibbs2002choosing}.}

\medskip
{
At this point, it is worth mentioning the following two standard facts concerning the relation between the Wasserstein distance $d_W(X,Y)$ and the so-called {\it Kolmogorov distance}
$$
d_K(X,Y) := \sup_{x\in \R} \left| \P[X\leq x] - \P[Y \leq x] \right| dx.
$$
\begin{enumerate}
\item[(a)] If $Y$ has a density bounded by some constant $A\in (0,\infty)$ and $X$ is integrable, then one has that 
$$
d_K(X,Y) \leq C \sqrt{d_W(X,Y)},
$$
where $C$ is a constant possibly depending on $A$ (one can obtain such an estimate e.g. by mimicking the proof of \cite[Theorem 3.3]{chen2010normal}).
\item[(b)] If $X$ and $Y$ take values in $[0,1]$, then
$$
d_W(X,Y) =  \int_{[0,1]}  \big| \P[X \leq x] - \P[Y\leq x]\big| \;dx \leq d_K(X,Y).
$$
\end{enumerate}
In particular, the estimates appearing in (a) and (b) may be combined with Theorem \ref{thm:main}, in order to deduce (arguably not optimal) upper and lower bounds on the rate of convergence in the Kolmogorov distance for the limit theorem in \eqref{eq:2}.
}

Since the quantities of interest here only involve the distribution of the considered random variables, we make the following assumption in the rest of the paper : on some probability space $(\Omega,\mathcal F, \P)$, we are given a random variable $\theta$ with values in $[0,1]$ and with law denoted $\mu$, and a sequence $(X_k)_{k \geq 1}$ of 0-1 random variables such that $(X_k)_{k\geq 1}$ are, conditionally on $\theta$, i.i.d. Bernoulli random variables with parameter $\theta$. The distribution of $(X_k)_{k \geq 1}$ is then given by \eqref{eq:definetti}.

For a sequence of random variables $(V_k)_{k\geq 1}$, we write $$\bar V_n = n^{-1}\sum_{k=1}^n V_k.$$
We also adopt the following notation : for two real-valued non-negative sequences $(a_n)_{n\geq 1}$ and $(b_n)_{n\geq 1}$, we write $a_n \cong b_n$ if both $a_n = \mO (b_n)$ and $b_n = \mO(a_n)$.

\section{Bounds in the general case}
\label{sec:boundsgeneral}
\subsection{Preliminaries}

We start with the following simple proposition, which shows that $\was(\bar X_n, \theta)$ is bounded from above and below, respectively, by terms of the order $1/\sqn$ and $1/n$.
\begin{prop}
\label{prop:boundbelow}
It holds that
\eq
\frac{\E[\theta(1-\theta)]}{n}\leq \was\left(\bar{X_n}, \theta\right) \leq  \sqrt\frac{\E[\theta(1-\theta)]}{n}.
\qe
\end{prop}
\begin{proof}
The Cauchy-Schwarz inequality implies
\begin{align*}
\Ecro{|\bar X_n - \theta|} &\leq \sqrt {\Ecro{(\bar X_n - \theta)^2}} \\
& =\sqrt{ \Ecro{\Ecro{(\bar X_n - \theta)^2\; | \; \theta}}}\\
& = \sqrt {\frac{\E[\theta(1-\theta)]}{n}},
\end{align*}
giving the upper bound.

To show the lower bound, we use the dual formulation of the Wasserstein distance. First we remark that  the function $\psi \; : \; x \mapsto x(x-1) \mathbf 1_{[0,1]}(x)$ is 1-Lipschitz. We have
\begin{align*}
\Ecro{\psi(\bar X_n)}= \Ecro{\bar X_n(\bar X_n-1)}&= \Ecro{\bar X_n^2} - \Ecro{\bar X_n}\\
&= \Ecro{ \Ecro{\bar X_n^2\, | \, \theta }}- \E[\theta] \\
&= \Ecro{  \frac{\theta(1-\theta)}{n} + \theta^2 }- \E[\theta] .
\end{align*}
Thus,
\begin{align*}
\Ecro{\psi(\bar X_n)} - \Ecro{\psi(\theta)} = \frac{\Ecro{  \theta(1-\theta) }}{n}. 
\end{align*}
From \eqref{eq:was2}, we get the desired result.
\end{proof}

From the previous Proposition, if $\mu(\{0,1\}) =1$ (or, equivalently, $\theta =0$ or $\theta =1$ almost surely), then
$$\was(\bar X_n,\theta) = 0.$$

Our next lemma shows that, if $\mu(\{0,1\}) <1$, then the rate of convergence of $\was(\bar X_n, \theta)$ does not change if we ``kill'' the mass of $\mu$ on $\{0,1\}$.
\begin{lma}
\label{lem:bords}
Assume that $\mu(\{0,1\}) < 1$. Let $\tilde \theta$ have the law $\tilde \mu$ defined by $\tilde \mu(A) = \mu( A \backslash \{0,1\} )/(1-\mu(\{0,1\}) $ for all Borel sets $A$ of $\R$, and $(Y_k)_{k\geq 1}$ be a Bernoulli sequence with prior $\tilde \theta$. Then $\tilde \mu(\{0,1\}) = 0$ and
\eq
\label{eq:lembords}
\was(\bar X_n, \theta) = (1-\mu(\{0,1\})\; \was(\bar Y_n, \tilde \theta).  \qe In particular,
$$\was(\bar X_n, \theta) \cong \was(\bar Y_n, \tilde \theta).$$
\end{lma}

\begin{proof}
Let $\psi \in \lip(1)$. Then
\begin{align*}
\E[\psi(\bar X_n)] -\E[ \psi(\theta)]=& \E[\psi(\bar X_n)\mathbf 1_{\{\theta = 0\}}] -\E[ \psi(\theta)\mathbf 1_{\{\theta = 0\}}]\\
+&\E[\psi(\bar X_n)\mathbf 1_{\{\theta = 1\}}] -\E[ \psi(\theta)\mathbf 1_{\{\theta = 1\}}]\\
+&\E[\psi(\bar X_n)\mathbf 1_{\{\theta \neq 0 \text{ and } \theta\neq 1\}}] -\E[ \psi(\theta)\mathbf 1_{\{\theta \neq 0 \text{ and } \theta\neq 1\}}]\\
=&\E[\psi(\bar X_n)\mathbf 1_{\{\theta \neq 0 \text{ and } \theta\neq 1\}}] -\E[ \psi(\theta)\mathbf 1_{\{\theta \neq 0 \text{ and } \theta\neq 1\}}],
\end{align*}
since, both on $\{ \theta = 0 \}$ and $\{\theta =1\}$, $\bar X_n = \theta$ a.s. However, from the very definition of $(Y_k)_{k \geq 1}$, 
\begin{align*}
\E[\psi(\bar X_n)\mathbf 1_{\{\theta \neq 0 \text{ and } \theta\neq 1\}}]  &= \int_{(0,1)} \E[\psi(\bar X_n \; | \; \theta = t)] \mu(dt) \\
&= \mu( (0,1) ) \int_{(0,1)} \E[\psi(\bar X_n) \; | \; \theta = t] \frac{\mu(dt)}{\mu((0,1))}\\
&=(1-\mu(\{0,1\})) \int_{(0,1)} \E[\psi(\bar Y_n) \; | \; \theta = t] \frac{\mu(dt)}{\mu((0,1))}\\
&= (1-\mu(\{0,1\})) \E[\psi(\bar Y_n)].
\end{align*}
A similar argument for $\E[ \psi(\theta)\mathbf 1_{\{\theta \neq 0 \text{ and } \theta\neq 1\}}]$ and taking the supremum over all 1-Lipschitz functions $\psi$ gives the desired result.
\end{proof}
From now on, we assume that $\mu (\{0,1\}) = 0$ (equivalently, $0<\theta<1$ a.s.). {Exchangeable sequences such that the associated de Finetti measure has support contained in $(0,1)$ are sometimes called {\it non-deterministic} --- see e.g. \cite{edp, edpp, hls, p}. }

\medskip

\subsection{Equivalent formulation with a perturbed version of the prior}

In this section, we show that the problem of bounding the Wasserstein distance between $\bar X_n$ and $\theta$ is equivalent in some sense to bounding the Wasserstein distance between $\theta$ and some perturbed version of $\theta$. Recall that we assume $\mu (\{0,1\}) = 0$. We will make use of a theorem giving a Berry-Ess\'een type bound in Wasserstein distance in the classical Central Limit Theorem for Bernoulli random variables  (a proof can be found in \cite{chen2010normal}, Corollary 4.1). We quote it here.
\begin{thm}(Chen, 2005)
\label{thm:chen}
Let $(V_k)_{k\geq1}$ be a sequence of i.i.d.\ Bernoulli random variables with parameter $t \in (0,1)$. Let $Y_k = (t(1-t))^{-1/2} (V_k -t)$ (so that $Y_k$ has mean 0 and variance 1), and let $\bar Y_n = n^{-1} \sum_{k=1}^n Y_k$. Then
\eq
\label{eq:berryEss\'een}
\was(\sqn \bar Y_n,Z) \leq \frac{\E[|Y_1|^3}{\sqn} = \frac{t^2 + (1-t)^2}{\sqrt{n\, t(1-t)}},
\qe
where $Z$ stands for a standard normal random variable.
\end{thm}
The main result of this section is the following proposition.
\begin{prop}
\label{prop:equivalence}
Let $Z$ stand for a standard normal random variable independent of $\theta$. Then
\eq
\label{eq:bound}
\bigg| \was (\bar X_n , \theta ) - \was \left(\theta + \sqrt{\frac{\theta(1-\theta)}{n}} Z\; , \; \theta\right)  \bigg| \leq  \frac{\E[\theta^2 + (1-\theta)^2]}{n}.
\qe
\end{prop}
\begin{proof}
Let $\psi \in \lip(1)$. For every $t \in (0,1)$, we define the function $\phi_t$ by $\phi_t(\sqrt{n} (t(1-t))^{-1/2} (x- t) ) = \psi(x)$, or equivalently,
$$\phi_t(x) = \psi\left(t+\sqrt\frac{t(1-t)}{n} x\right).$$
Clearly we have $\phi_t \in \lip\left( \sqrt\frac{t(1-t)}{ n} \right)$.

Let $Y_i = (\theta(1-\theta))^{-1/2} (X_i-\theta)$. We have
\begin{align*}
  &\Ecro{\psi(\bar X_n )} - \E[\psi(\theta)]\\
& = \Ecro{ \Ecro{\psi(\bar X_n ) \; | \; \theta}} - \E[ \psi(\theta)]\\
& = \Ecro{\Ecro{\phi_\theta (\sqrt{n} (\theta(1-\theta))^{-1/2} (\bar X_n- \theta) ) \; | \; \theta} } - \E[\psi(\theta)]\\
& = \Ecro{\Ecro{\phi_\theta(\sqrt{n} \bar Y_n ) \; | \; \theta}} - \E[\psi(\theta)].
\end{align*}
Let $Z$ be a standard normal variable independent of $\theta$. Theorem \ref{thm:chen} together with the fact that $\phi_t \in \lip\left( \sqrt\frac{t(1-t)}{ n} \right)$ implies
\begin{align*}\Ecro{\phi_\theta(\sqrt{n} \bar Y_n ) \; | \; \theta} &\leq \Ecro{\phi_\theta(Z) \; | \; \theta} + \frac{\theta^2 +(1-\theta)^2}{\sqrt{n\theta(1-\theta)}}  \cdot \sqrt\frac{\theta(1-\theta)}{ n}\\
&= \Ecro{\phi_\theta(Z) \; | \; \theta} + \frac{\theta^2 +(1-\theta)^2}{n}.
\end{align*}
Thus,
\begin{align*}
  &\Ecro{\psi(\bar X_n )} - \E[\psi(\theta)]\\
& \leq  \Ecro{\Ecro{\phi_\theta(Z)\; |\; \theta}} - \E[\psi(\theta)] + \frac{\E[\theta^2 +(1-\theta)^2]}{n}\\
& =\Ecro{\psi\left(\theta+\sqrt\frac{\theta(1-\theta)}{n} Z\right)} - \E[\psi(\theta)] + \frac{\E[\theta^2 +(1-\theta)^2]}{n}\\
& \leq  \was\left( \theta, \theta+\sqrt\frac{\theta(1-\theta)}{n} Z \right) +  \frac{\E[\theta^2 +(1-\theta)^2]}{n}.
\end{align*}
Taking the supremum over all $\psi \in \lip(1)$, we get
$$\was (\bar X_n , \theta ) \leq \was \left(\theta + \sqrt{\frac{\theta(1-\theta)}{n}} Z\; , \; \theta\right) + \frac{\E[\theta^2 + (1-\theta)^2]}{n}.$$
In a similar way, one can show that $$\was \left(\theta + \sqrt{\frac{\theta(1-\theta)}{n}} Z\; , \; \theta\right) \leq \was (\bar X_n , \theta )  +\frac{\E[\theta^2 + (1-\theta)^2]}{n}.$$ This completes the proof. 
\end{proof}


The same argument as in the proof of Proposition \ref{prop:boundbelow} shows that 
$$\was \left(\theta + \sqrt{\frac{\theta(1-\theta)}{n}} Z\; , \; \theta\right) \geq \frac{C}{n},$$
for some $C>0$. This together with Propositions \ref{prop:boundbelow} and \ref{prop:equivalence} leads to the next corollary.
\begin{cor}
\label{cor:equiv}
 If $\mu(\{0,1\}) = 0$, and if $Z$ stands for a standard normal random variable independent of $\theta$, then
$$\was(\bar X_n, \theta) \cong \was \left(\theta + \sqrt{\frac{\theta(1-\theta)}{n}} Z\; , \; \theta\right). $$
\end{cor}

We are left with the following question : given $Z$ a standard normal random variable independent of $\theta$, how does the quantity
$$ \was \left(\theta + \sqrt{\frac{\theta(1-\theta)}{n}}  Z\; , \; \theta\right) $$
behave as $n$ tends to infinity? To our knowledge, this kind of question has not yet been investigated in the literature. The answer is non-trivial and heavily depends on the law of $\theta$. For instance, when $\theta$ has the Beta distribution, we know from \cite{goldstein2013stein} (and Corollary  \ref{cor:equiv}) that $\was \left(\theta + \sqrt{\frac{\theta(1-\theta)}{n}}  Z\; , \; \theta\right) \cong 1/n$. As we will see in section \ref{sec:examples}, this is not true in general, even if $\theta$ has a density with respect to the Lebesgue measure.



However, in the next section, we show that $\was(\bar X_n,\theta) \cong 1/n$ whenever $\theta$ has a smooth density whose derivative satisfies some integrability property. This includes the case of the Beta distribution.

\section{Bounds in the case of a smooth density}
\label{sec:smooth}
\subsection{A general bound}
The main result of this section is the following Proposition.
\begin{prop}
\label{prop:main}
Assume the law $\mu$ of $\theta$ has a smooth density $p$ on $(0,1)$ satisfying $\int_0^1 u(1-u) |p'(u)| du<\infty$. Let $Z$ be a standard normal random variable independent of $\theta$. Then
\eq
\label{eq:boundpropmain}
\was \left(\theta + \sqrt{\frac{\theta(1-\theta)}{n}} Z\; , \; \theta\right)  \leq \frac{C(\mu)}{n},
\qe
where
$$C(\mu) = \int_0^1 |1-2u| p(u) du + \int_0^1 u(1-u) |p'(u)| du  + \frac{3}{\sqrt{2\pi e}}.$$
\end{prop}
\begin{proof}
The proof is rather calculatory and relies on the representation \eqref{eq:was4} of the Wasserstein distance. Let us give some notations first.
\begin{itemize}
\item $f(x) = \sqrt{x(1-x)}$, $x \in (0,1)$.
\item $\omega := x \mapsto \frac{1}{\sqrt{2\pi}} e^{-x^2/2}$ is the probability distribution function of a standard normal random variable.
\item For a real-valued random variable $X$, $F_X$ denotes its cumulative distribution function.
\item  $\forall t\geq 0$, $G(t) = 1-F_Z(t) = F_Z(-t)= \int_t^{+\infty} \omega(x) dx$.
\end{itemize}
We have
\eq
\label{eq:integralewasserstein}
\was\left(\theta + \frac{f(\theta)}{\sqrt n} Z\; , \; \theta\right) = \int_{\R} \bigg| \Pcro{\theta +\frac{f(\theta)}{\sqrt n} Z \leq x} - \Pcro{\theta \leq x} \bigg| dx
\qe
and
\begin{align}
\label{eq:1}
  &\Pcro{\theta +\sqrt{\frac{\theta(1-\theta)}{n}} Z \leq x} - \Pcro{\theta \leq x} \\
& = \int_0^1 \left(\Pcro{t+\frac{f(t)}{\sqn}Z \leq x} - \mathbf 1_{t\leq x}\right)p(t) dt\\
& = \int_0^1 \left(F_Z\left( \frac{\sqn}{f(t)}(x-t) \right) - \mathbf 1_{t\leq x}\right)p(t) dt.
\end{align}

We  split the integral \eqref{eq:integralewasserstein} in several parts, according to the range of $x$. 

\

\noindent {\bf Case 1 : $0 \leq x \leq 1/2$.} In this case we write
\begin{align*}
  &\Pcro{\theta +\sqrt{\frac{\theta(1-\theta)}{n}} Z \leq x} - \Pcro{\theta \leq x} \\
& = \int_0^x \left(F_Z\left( \frac{\sqn}{f(t)}(x-t) \right) - 1\right)p(t) dt + \int_x^1 F_Z\left( \frac{\sqn}{f(t)}(x-t) \right)p(t) dt\\
& = - \int_0^x G\left( \frac{\sqn}{f(t)}(x-t) \right) p(t) dt + \int_x^1 G\left( \frac{\sqn}{f(t)}(t-x) \right) p(t) dt\\
& = - \int_0^x G\left( \frac{\sqn}{f(x-t)}t \right) p(x-t) dt + \int_0^{1-x} G\left( \frac{\sqn}{f(t+x)} t \right) p(t+x) dt\\
 &= \int_0^x \left[ G\left( \frac{\sqn}{f(x+t)}t \right) p(x+t)-G\left( \frac{\sqn}{f(x-t)}t \right) p(x-t)  \right]dt \\
&\quad + \int_x^{1-x} G\left( \frac{\sqn t}{f(t+x)} \right) p(t+x) dt\\
& =\frac{1}{\sqn}  \int_0^{x\sqn} \left[ G\left( \frac{t}{f\left(x+\frac{t}{\sqn}\right)} \right) p\left(x+\frac{t}{\sqn}\right)-G\left( \frac{t}{f\left(x-\frac{t}{\sqn}\right)} \right) p\left(x-\frac{t}{\sqn}\right)  \right]dt  \\
&\quad + \int_{x}^{1-x} G\left( \frac{\sqn t}{f(x+t)}  \right) p(x+t) dt
\end{align*}
Define, for $(u,t) \in (0,1)\times(0,+\infty)$,
$$H(u,t) =  G\left( \frac{t}{f(u)}\right) p(u).$$
The function $H$ has a derivative with respect to its first argument and a direct computation yields
$$\partial_1 H(u,t) = t \frac{f'(u)}{f^2(u)} \omega\left(\frac{t}{f(u)}\right) p(u) + G\left(\frac{t}{f(u)}\right) p'(u).$$
Thus,
\begin{align*}
&\Pcro{\theta +\sqrt{\frac{\theta(1-\theta)}{n}} Z \leq x} - \Pcro{\theta \leq x} \\
& =\frac{1}{\sqn}   \int_0^{x\sqn} \left[ H\left(x+\frac{t}{\sqn}, t\right) - H\left(x-\frac{t}{\sqn},t\right)  \right]dt   +  \int_{x}^{1-x} G\left( \frac{\sqn t}{f(x+t)}  \right) p(x+t) dt\\
& =\frac{1}{\sqn}   \int_0^{x\sqn} \int_{x-\frac{t}{\sqn}}^{x+\frac{t}{\sqn}} \partial_1 H(u,t) du \, dt  +  \int_{x}^{1-x} G\left( \frac{\sqn t}{f(x+t)}  \right) p(x+t) dt\\
& = \frac{1}{\sqn} \bigg[  \int_0^{x\sqn}  \int_{x-\frac{t}{\sqn}}^{x+\frac{t}{\sqn}}  t\frac{f'(u)}{f^2(u)} \omega\left(\frac{t}{f(u)}\right) p(u) du\, dt \\
&\quad +  \int_0^{x\sqn} \int_{x-\frac{t}{\sqn}}^{x+\frac{t}{\sqn}}  G\left(\frac{t}{f(u)}\right) p'(u) du \, dt\bigg] \\
&\quad +  \int_{x}^{1-x} G\left( \frac{\sqn t}{f(x+t)}  \right) p(x+t) dt\\
& := A_1(x) + A_2(x) + A_3(x)
\end{align*}
We will bound seperately the integrals of the absolute values of $A_1$, $A_2$ and $A_3$ on $(0,1/2)$.

First we focus on $A_1$.
\begin{align*}
\int_0^{1/2} |A_1(x) |dx &= \frac{1}{\sqn} \int_0^{1/2} \bigg| \int_0^{x\sqn}  \int_{x-\frac{t}{\sqn}}^{x+\frac{t}{\sqn}} t  \frac{f'(u)}{f^2(u)} \omega\left(\frac{t}{f(u)}\right) p(u) \,du\, dt  \bigg| dx\\
&\leq \frac{1}{\sqn} \int_0^{1/2} \int_0^{x\sqn}  \int_{x-\frac{t}{\sqn}}^{x+\frac{t}{\sqn}} t  \frac{|f'(u)|}{f^2(u)} \omega\left(\frac{t}{f(u)}\right) p(u) \,du\, dt \, dx
\end{align*}
We apply Fubini's theorem with a (possibly) larger region of integration, using the fact that
\begin{align*}
&\left\{(x,t,u)\; | \; 0\leq x\leq\frac{1}{2}, \, 0\leq t\leq x\sqn, \, x-\frac{t}{\sqn}\leq u \leq x+\frac{t}{\sqn} \right\}\\
 \subset & \left\{(x,t,u)\; | \; 0\leq u\leq 1, \, 0\leq t\leq \frac{\sqn}{2}, \, u-\frac{t}{\sqn}\leq x \leq u+\frac{t}{\sqn} \right\}.
\end{align*}
This yields
\begin{align*}
\int_0^{1/2} |A_1(x) |dx&\leq \frac{1}{\sqn} \int_0^{1} p(u) \frac{|f'(u)|}{f^2(u)}  \int_0^{\sqn/2}  \int_{u-\frac{t}{\sqn}}^{u+\frac{t}{\sqn}}   dx\; t \;\omega\left(\frac{t}{f(u)}\right) \, du \,dt\\
&= \frac{2}{n} \int_0^1  p(u) \frac{|f'(u)|}{f^2(u)}  \int_0^{\sqn/2} t^2 \omega\left(\frac{t}{f(u)}\right)  \,dt\, du\\
&= \frac{2}{n} \int_0^1 p(u)\frac{|f'(u)|}{f^2(u)} f^{3}(u) \int_0^{\frac{\sqn}{2f(u)}} t^2 \omega(t)  \,dt\, du\\
&\leq \frac{1}{n} \int_0^1 p(u)| (f^2(u))' | \int_0^{+\infty} t^2 \omega(t) \,dt\, du\\
&= \frac{1}{2n} \int_0^1 p(u)| 1-2u |du\\
\end{align*}

A similar computation yields
\begin{align*}
\int_0^{1/2} |A_2(x) |dx &= \frac{1}{\sqn} \int_0^{1/2} \bigg| \int_0^{x\sqn}  \int_{x-\frac{t}{\sqn}}^{x+\frac{t}{\sqn}}  G\left(\frac{t}{f(u)}\right) p'(u) \,du\, dt  \bigg| dx\\
&\leq \frac{1}{\sqn} \int_0^1  |p'(u)|  \int_0^{\sqn/2} \int_{u-\frac{t}{\sqn}}^{u+\frac{t}{\sqn}}   dx \; G\left(\frac{t}{f(u)}\right)  \,dt\, du\\
&= \frac{2}{n} \int_0^1  |p'(u)|   \int_0^{\sqn/2} t \,G\left(\frac{t}{f(u)}\right)  \,dt\, du\\
&= \frac{2}{n} \int_0^1 |p'(u)|f^{2}(u) \int_0^{\frac{\sqn}{2f(u)}} t\, G(t)  \,dt\, du\\
&\leq \frac{2}{n} \int_0^1 |p'(u)| u(1-u) \int_0^{+\infty} t \,G(t) \,dt\, du\\
&= \frac{1}{2n} \int_0^1 |p'(u)| u(1-u) du,
\end{align*}
where we used the fact that $\int_0^{+\infty} t \,G(t) \,dt = 1/4$ (for instance from an integration by parts).

As for $A_3$, using Fubini's theorem again we have
\begin{align*}
\int_0^{1/2} |A_3(x)| dx &= \int_0^{1/2} \int_x^{1-x} G\left( \frac{\sqn}{f(t+x)} t \right) p(t+x)\, dt\,dx\\
&= \int_0^{1/2} \int_{2x}^{1} G\left( \frac{\sqn}{f(t)} (t-x) \right) p(t)\, dt\,dx\\
&= \int_0^{1} p(t) \int_{0}^{t/2} G\left( \frac{\sqn}{f(t)} (t-x) \right) \, dx\,dt\\
&= \int_0^{1} p(t) \int_{t/2}^{t} G\left( \frac{\sqn}{f(t)} x \right) \, dx\,dt\\
&= \frac{1}{\sqn}\int_0^{1} p(t) f(t) \int_{\frac{\sqn t}{2f(t)}}^{\frac{\sqn t}{f(t)}}G(x) \,dx\, dt\\
& \leq \frac{1}{\sqn}\int_0^{1} p(t) f(t) \int_{\frac{\sqn t}{2f(t)}}^{+\infty}G(x) \,dx\, dt
\end{align*}
Now, integrating by parts we have $\int_y^{+\infty} G(u) du = \omega(y)-yG(y) \leq \omega(y)$, so that 
\eq
\label{eq:G}
\forall y >0, \quad \int_y^{+\infty} G(u) du \leq \omega(y) \leq \frac{1}{y\sqrt{2\pi e}},
\qe
an inequality  easily shown for instance by studying the function $y\mapsto y\,\omega(y)$. This yields
\begin{align*}
\int_0^{1/2} |A_3(x)| dx \leq \frac{2}{n\sqrt{2\pi e}} \int_0^1 p(t) \frac{f^2(t)}{t} dt  = \frac{\sqrt 2}{n\sqrt{\pi e}} \int_0^1 p(t) (1-t)\, dt
\end{align*}
To sum up,
\eq
\label{eq:bound3}
\begin{aligned}
&\int_0^{1/2} \bigg|\Pcro{\theta +\sqrt{\frac{\theta(1-\theta)}{n}} Z \leq x} - \Pcro{\theta \leq x} \bigg| dx\\
& \leq  \frac{1}{2n} \bigg[ \int_0^1 p(u)|1-2u| du + \int_0^1 u(1-u)|p'(u)|du +  \frac{2\sqrt 2}{\sqrt{\pi e}} \int_0^1 p(u) (1-u) du \bigg].
\end{aligned}
\qe

\

\noindent{\bf Case 2 : $x \leq 0$.} In this case, from \eqref{eq:1} we have
$$\Pcro{\theta +\sqrt{\frac{\theta(1-\theta)}{n}} Z \leq x} - \Pcro{\theta \leq x} = \int_0^1 G\left( \frac{\sqn}{f(t)}(t-x) \right) p(t) dt,$$
so that
\begin{align*}
&\int_{-\infty}^0 \bigg| \Pcro{\theta +\sqrt{\frac{\theta(1-\theta)}{n}} Z \leq x} - \Pcro{\theta \leq x}  \bigg | \; dx\\
& = \int_{-\infty}^0 \int_0^1 G\left( \frac{\sqn}{f(t)}(t-x) \right) p(t) \,dt \,dx \\
& = \int_{0}^{+\infty} \int_0^1 G\left( \frac{\sqn}{f(t)}(t+x) \right) p(t) \,dt \,dx \\
& = \int_0^1 p(t) \int_{0}^{+\infty} G\left( \frac{\sqn}{f(t)}(t+x) \right) \,dx \,dt\\
& =  \frac{1}{\sqn} \int_0^1 p(t) f(t) \int_{\frac{\sqn t}{f(t)}}^{+\infty} G(u) \,du\, dt.
\end{align*}
Using \eqref{eq:G}, we obtain
\eq
\label{eq:bound4}
\begin{aligned}
&\int_{-\infty}^0 \bigg| \Pcro{\theta +\sqrt{\frac{\theta(1-\theta)}{n}} Z \leq x} - \Pcro{\theta \leq x}  \bigg | \; dx\\
&  \leq  \frac{1}{n \sqrt{2\pi e}} \int_0^1 p(t) \frac{f^2(t)}{t} dt = \frac{1}{n \sqrt{2\pi e}} \int_0^1 p(t)(1-t) dt.
\end{aligned}
\qe
From  \eqref{eq:bound3} and \eqref{eq:bound4} we get
\eq
\label{eq:bound2}
\begin{aligned}
&\int_{-\infty}^{1/2} \bigg|\Pcro{\theta +\sqrt{\frac{\theta(1-\theta)}{n}} Z \leq x} - \Pcro{\theta \leq x} \bigg| dx\\
& \leq \frac{1}{2n} \bigg[ \int_0^1 p(u)|1-2u| du + \int_0^1 u(1-u)|p'(u)|du  +  \frac{3\sqrt 2}{\sqrt{\pi e}} \int_0^1 p(u) (1-u) du\bigg].
\end{aligned}
\qe

\

\noindent {\bf Case 3 : $x \geq 1/2$}. In this case, we can use the symmetry of $f$ and $Z$ and the bound found for $x \leq 1/2$. More precisely, let $\theta_1 = 1- \theta$. Then, since $f(1-t) = f(t)$,
\begin{align*}
&\Pcro{\theta + \frac{f(\theta)}{\sqn} Z \leq x } - \Pcro{\theta\leq x}\\
& =\Pcro{1-\theta - \frac{f(1-\theta)}{\sqn} Z \geq 1-x } - \Pcro{1-\theta\geq 1-x}\\
& =\Pcro{\theta_1 - \frac{f(\theta_1)}{\sqn} Z \geq 1-x } - \Pcro{\theta_1\geq 1-x}\\
& =\Pcro{\theta_1 + \frac{f(\theta_1)}{\sqn} Z \geq 1-x } - \Pcro{\theta_1\geq 1-x}\\
& =\Pcro{\theta_1\leq 1-x}-\Pcro{\theta_1 + \frac{f(\theta_1)}{\sqn} Z \leq 1-x } .
\end{align*}
Thus,
\begin{align*}
&\int_{1/2}^{+\infty} \bigg|\Pcro{\theta + \frac{f(\theta)}{\sqn} Z \leq x } - \Pcro{\theta\leq x} \bigg| dx\\
& =\int_{1/2}^{+\infty} \bigg|\Pcro{\theta_1 + \frac{f(\theta_1)}{\sqn} Z \leq 1-x } - \Pcro{\theta_1\leq 1-x}\bigg| dx\\
& =\int_{-\infty}^{1/2} \bigg|\Pcro{\theta_1 + \frac{f(\theta_1)}{\sqn} Z \leq x } - \Pcro{\theta_1\leq x}\bigg| dx.
\end{align*}
Now we can use the bound in \eqref{eq:bound2} with the transformation $p(u) \rightarrow p(1-u)$, to obtain (after a change of variables $v=1-u$ in the integrals)~:
\begin{align*}
&\int_{1/2}^{+\infty} \bigg|\Pcro{\theta +\sqrt{\frac{\theta(1-\theta)}{n}} Z \leq x} - \Pcro{\theta \leq x} \bigg| dx\\
& \leq \frac{1}{2n} \left[ \int_0^1 p(u)|1-2u| du + \int_0^1 u(1-u)|p'(u)|du +  \frac{3\sqrt 2}{\sqrt{\pi e}} \int_0^1 p(u) u du\right].
\end{align*}
The proof follows from the last inequality and \eqref{eq:bound2}.
\end{proof}

\begin{rmk}
If  $p$ does not vanish on $(0,1)$ and $\rho_\theta(x) =\frac{p'(x)}{p(x)}$ is the score function of $\theta$, then the condition $\int_0^1 u(1-u) |p'(u)|du<\infty$ can be rewritten as
$$\Ecro{\theta(1-\theta)|\rho_\theta(\theta)|}<\infty. $$
\end{rmk}

The proof of our main theorem follows easily.

\begin{proof}[Proof of Theorem \ref{thm:main}] The first inequality in \eqref{eq:ineqmain} is just a restatement of Proposition \ref{prop:boundbelow}, whereas the upper bound follows 
from Propositions \ref{prop:equivalence} and \ref{prop:main}.
\end{proof}

\subsection{Application to the Beta distribution}

We specialize the result of Theorem \ref{thm:main} to the case of the Beta distribution. We explicit the bounds in \eqref{eq:ineqmain} when the density $p$ is given by \eqref{eq:densitebeta}. As a by-product, we obtain bounds of the optimal order with explicit constants for the distance of the scaled number of white balls drawn from a P\'olya-Eggenberger urn to its limiting distribution. As said before, such bounds were already obtained in \cite{goldstein2013stein}, with explicit constants as well. This will allow us to numerically compare the constants found in this article and the ones in \cite{goldstein2013stein}.

We begin with a Lemma which can be shown by elementary computations. Recall $B$ denotes the Beta function and we denote by $B_i$ the incomplete Beta function : for $x \in [0,1], \alpha>0,\beta>0$, $B_i(x,\alpha,\beta) = \int_0^x t^{\alpha-1}(1-t)^{\beta-1} dt$.
\begin{lma}
\label{lem:F}
Let $\theta$ have the Beta distribution with parameters $\alpha$ and $\beta$. For $a,b\in \R$ let 
$$F(\alpha,\beta,a,b) = \Ecro{|a \theta +b|}.$$
 Then, if $a>0$,
\begin{align*}
&F(\alpha,\beta,a,b)  \\
=&\left\{ \begin{array}{ll}
 \frac{1}{B(\alpha,\beta)}\big[ -2aB_i(-b/a,\alpha+1,\beta)\\
\quad-2bB_i(-b/a,\alpha,\beta)+  aB(\alpha+1,\beta)+bB(\alpha,\beta) \big] &\text{if} \; -a < b \leq 0,\\
 \frac{1}{B(\alpha,\beta)}\left[ aB(\alpha+1,\beta)+bB(\alpha,\beta)\right]&\text{if} \; b>0,\\
  -\frac{1}{B(\alpha,\beta)}\left[ aB(\alpha+1,\beta)+bB(\alpha,\beta)\right]  &\text{if} \; b\leq -a.
\end{array} \right.
\end{align*}
If $a=0$, $F(\alpha,\beta,0,b) = |b|$, and if $a<0$, $F(\alpha,\beta,a,b) = F(\alpha,\beta,-a,-b)$.
\end{lma}
\begin{proof}
If $-a< b \leq 0$, then $0\leq-\frac{b}{a}< 1$, so that
\begin{align*}
 \Ecro{|a \theta +b|} = &\frac{1}{B(\alpha,\beta)}\bigg[ \int_0^{-b/a} (-a t - b ) t^{\alpha-1}(1-t)^{\beta-1} dt +  \int_{-b/a}^{1} (a t + b ) t^{\alpha-1}(1-t)^{\beta-1} dt  \bigg].
\end{align*}
An expansion and straightforward calculations give the result in this case. The other cases are dealt with similarly.
\end{proof}

\begin{prop}
Assume that the law $\mu$ of $\theta$ is the Beta distribution with parameters $\alpha$ and $\beta$. Then
\eq
\was(\bar X_n,\theta ) \leq \frac{C_{\alpha,\beta}}{n},
\qe
where
\eq
\label{eq:calphabeta}
C_{\alpha,\beta}=\frac{B(\alpha+2,\beta)+B(\alpha,\beta+2)}{B(\alpha,\beta)}+F(\alpha,\beta,2,-1)+ F(\alpha,\beta,  \alpha+\beta-2, 1-\alpha )  + \frac{3}{\sqrt{2\pi e}},
\qe
and $F$ is defined in Lemma \ref{lem:F}.
\end{prop}
\begin{proof}
If $p$ is the density of the Beta distribution defined in \eqref{eq:densitebeta}, it is clear that $p$ satisfies the assumptions of Theorem \ref{thm:main}. Note that $\int_0^1 |1-2u| p(u) du = \Ecro{|1-2\theta|} = F(\alpha,\beta,2,-1)$. It is straightforward to show that $\int_0^1 (u^2+(1-u)^2) p(u) du =  \frac{B(\alpha+2,\beta)+B(\alpha,\beta+2)}{B(\alpha,\beta)}$.
Moreover,
\begin{align*}
  \int_0^{1} u(1-u) |p'(u)| du& =\int_0^1 u(1-u) \bigg| \frac{\alpha-1}{u}  - \frac{\beta-1}{1-u}\bigg| p(u) du\\
 &= \int_0^1  \bigg| (\alpha-1)(1-u)  - (\beta-1)u\bigg| p(u) du\\
 &= \int_0^1 |(\alpha+\beta-2) u +1-\alpha| p(u) du \\
 &= F(\alpha,\beta,  \alpha+\beta-2, 1-\alpha ),
\end{align*}
proving our claim.
\end{proof}

\begin{cor}
\label{cor:polya}
In a P\'olya-Eggenberg urn containing initially $A$ white balls and $B$ black balls, and where at each draw a ball is replaced along with $m$ balls of the same color, let $\bar X_n$ be the scaled number of white balls in $n$ draws. Let $\theta$ have the Beta distribution with parameters $A/m$ and $B/m$. Then
$$\was(\bar X_n, \theta) \leq \frac{C_{A/m,B/m}}{n},$$
where $C_{\alpha,\beta}$ is defined in \eqref{eq:calphabeta}.
\end{cor}

Now, let us compare this result with the one of Goldstein and Reinert. We plot the ratio of the constant $C_{A/m,B/m}$ to the one obtained in \cite{goldstein2013stein}, Theorem 1.1, for values of $A/m$ and $B/m$ ranging from $10^{-5}$ to $3$.  \begin{figure}[!h] \center\includegraphics[width=10cm]{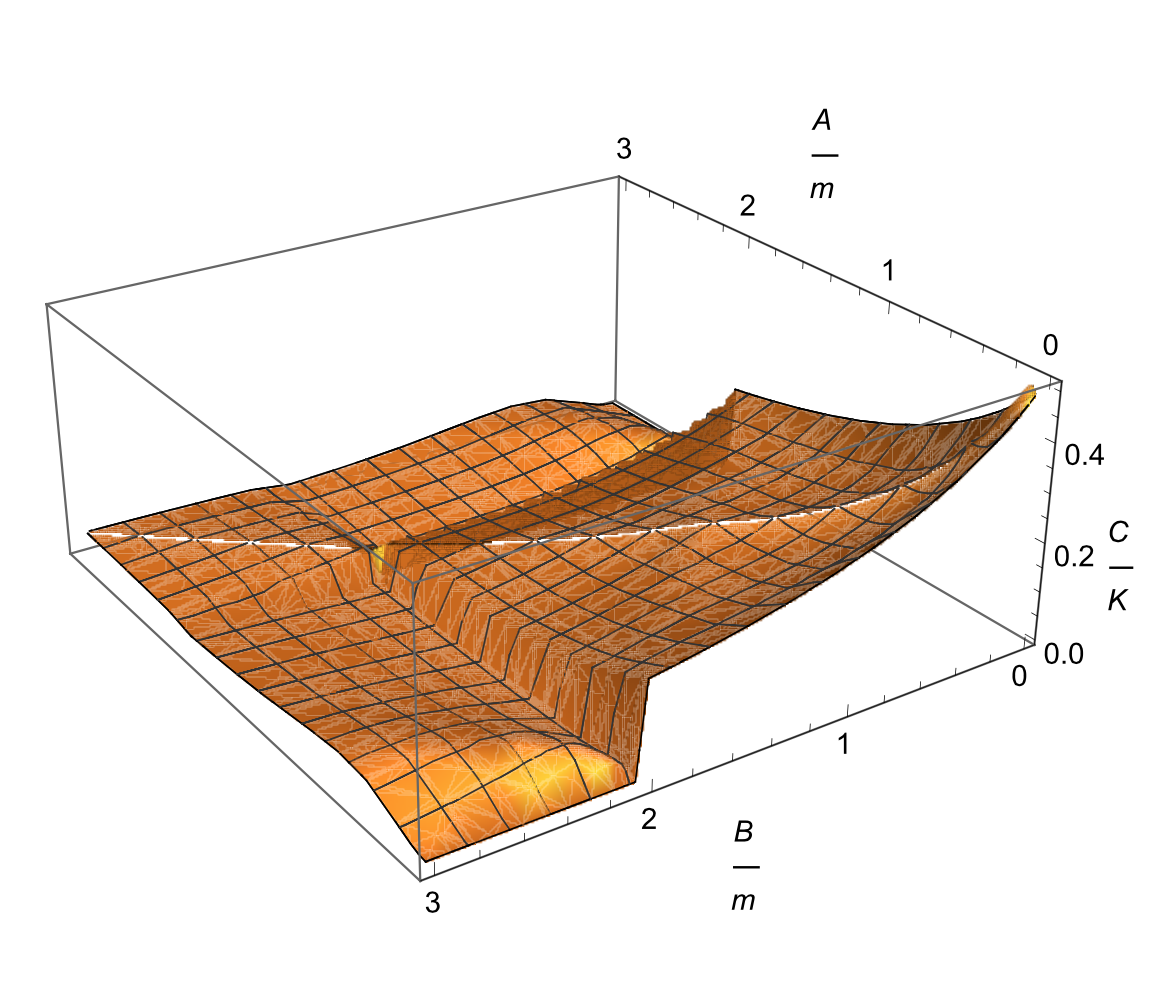} \caption{$C_{A/m,B/m}/K_{A/m,B/m}$ where $K_{A/m,B/m}$ is defined in \cite{goldstein2013stein}, Theorem 1.1. $(A/m,B/m) \in [10^{-5},3]^2$.}  \end{figure}

As we can see, our constant $C_{A/m,B/m}$ is at least half that of \cite{goldstein2013stein} for the set of parameters we chose; the ratio seems to go to zero as $A/m$ or $B/m$ become large.

\section{The rate $1/n^\delta$ is possible for any $1/2 \leq \delta \leq 1$}
\label{sec:examples}

We saw in Proposition \ref{prop:main} that a sufficient condition to get a rate of convergence of the order $1/n$ for $\was \left(\theta + \sqrt{\frac{\theta(1-\theta)}{n}} Z\; , \; \theta\right)$ is that $\mu$ is absolutely continuous with density $p$ on $(0,1)$ satisfying \linebreak $\int_0^1 u(1-u) |p'(u)| du<\infty$. The goal of this section is to show that this is not true anymore with the weaker asumption that $\mu$ is simply absolutely continuous. Actually, for each $\delta \in (1/2,1)$, we give an example of a measure $\mu$ with a density on $(0,1)$ such that $\was\left(\theta + \sqrt{\frac{\theta(1-\theta)}{n}} Z\; , \; \theta\right)$ is of the order $1/n^\delta$. This is the content of the two following propositions.

\begin{prop}
\label{prop:grando}
Let $\gamma \in (0,1)$ and suppose that $\theta$ has the density $p(x) = C_p (x-\frac{1}{2})^{\gamma-1} \mathbf 1_{(1/2,3/4)}(x)$ ($C_p$ is a normalizing constant). Let $(X_k)_{k\geq 1}$ be a Bernoulli sequence with prior $\theta$. Then
$$\was(\bar X_n,\theta) = \mO\left(  \frac{1}{n^{\frac{1+\gamma}{2}}} \right).$$
\end{prop}
\begin{proof}
From Proposition \ref{prop:equivalence}, it is sufficient to show that
$$\was \left(\theta + \sqrt{\frac{\theta(1-\theta)}{n}} Z\; , \; \theta\right)  = \mO \left( \frac{1}{n^{\frac{1+\gamma}{2}}}  \right),$$
where $Z$ stands for a normal random variable independent of $\theta$. The proof is similar to the one of Proposition \ref{prop:main}; we use here the same notations and do not give all the details in the calculations. In the following, $M_1$ and $M_2$ are generic positive constants that may vary from line to line.
We have
$$
\was\left(\theta + \frac{f(\theta)}{\sqrt n} Z\; , \; \theta\right) = \int_{\R} \bigg| \Pcro{\theta +\frac{f(\theta)}{\sqrt n} Z \leq x} - \Pcro{\theta \leq x} \bigg| dx.
$$
Define $\tth = 4(\theta-\frac{1}{2})$, so that $\tth$ has density equal to $\tilde C_p x^{\gamma-1} \mathbf 1_{(0,1)}(x)$ ($\tilde C_p$ normalizing constant). For $x \in (0,1)$, let $\tf(x) = f\left(1/2+\frac{x}{4}\right)$. We have
\begin{align*}
&\int_{\R} \bigg| \Pcro{\theta +\frac{f(\theta)}{\sqrt n} Z \leq x} - \Pcro{\theta \leq x} \bigg| dx\\
& =\int_{\R} \bigg| \Pcro{\tth +\frac{\tf(\tth)}{\sqrt n} Z \leq 4(x-1/2)} - \Pcro{\tth \leq 4(x-1/2)} \bigg| dx\\
& =\frac{1}{4}\int_{\R} \bigg| \Pcro{\tth +\frac{\tf(\tth)}{\sqrt n} Z \leq x} - \Pcro{\tth \leq x} \bigg| dx
\end{align*}
Thus it suffices to show that the last quantity is a $\mO\left(  \frac{1}{n^{\frac{1+\gamma}{2}}} \right)$.

\

\noindent  {\bf Case 1 : $0\leq x \leq 1/2$}. In this case
\begin{align}
&\frac{1}{\tilde C_p} \left( \Pcro{\tth +\frac{\tf(\tth)}{\sqrt n} Z \leq x} - \Pcro{\tth \leq x} \right) \notag \\
& =   \int_0^{x} \left( G\left( \frac{\sqn t}{\tf(x+t)} \right) (x+t)^{\gamma-1}-G\left( \frac{\sqn t}{\tf(x-t)} \right) (x-t)^{\gamma-1}  \right)dt \notag \\
&+ \int_{x}^{1-x} G\left( \frac{\sqn t}{\tf(x+t)}  \right) (x+t)^{\gamma-1} dt. \label{eq:moche}
\end{align}
If $\tilde H(u,t) =  G\left( \frac{t}{\tf(u)}\right) u^{\gamma-1}$, then
$$\partial_1 \tilde H(u,t) = t \frac{\tf'(u)}{\tf^2(u)} \omega\left(\frac{t}{\tf(u)}\right) u^{\gamma-1} - (1-\gamma)G\left(\frac{t}{\tf(u)}\right)u^{\gamma-2}.$$
 It is clear from the definition of $f$ and $\tf$ that $0<m_1\leq \tf(u) \leq 1$ and $| \tf'(u)| \leq m_2$ for some constants $m_1$ and $m_2$. Thus
 $$|\partial_1 \tilde H(u,t)| \leq M_1 t \,\omega(t) u^{\gamma-1} + M_2 G(t) u ^{\gamma-2}.$$
 We get
 \begin{align*}
&\int_0^{1/2} \int_0^{x} \bigg| G\left( \frac{\sqn t}{\tf(x+t)} \right) (x+t)^{\gamma-1}-G\left( \frac{\sqn t}{\tf(x-t)} \right) (x-t)^{\gamma-1}  \bigg| dt \, dx\\
& \leq \int_0^{1/2} \int_0^{x}   \int_{x-t}^{x+t} \left( M_1 \sqn\, t \,\omega(\sqn t) u^{\gamma-1} + M_2 G(\sqn t) u ^{\gamma-2} \right) du\, dt\, dx
\end{align*}
However,
\begin{align*}
&\int_0^{1/2} \int_0^{x}   \int_{x-t}^{x+t}  \sqn \,t \,\omega(\sqn t) u^{\gamma-1} du\, dt\, dx\\
& =\frac{1}{n^{1+\gamma/2}} \int_0^{\sqn/2} \int_0^{x}   \int_{x-t}^{x+t}   t \,\omega( t) u^{\gamma-1} du\, dt\, dx\\
& =\frac{1}{n^{1+\gamma/2}} \int_0^{\sqn/2} t\, \omega(t) \int_t^{\sqn/2}   \int_{x-t}^{x+t}      u^{\gamma-1} du\, dx\, dt\\
 & \leq \frac{1}{n^{1+\gamma/2}} \int_0^{\sqn/2} t \,\omega(t)   \int_0^{t+\sqn/2}  u^{\gamma-1}\int_{u-t}^{u+t} dx     dx\, du\, dt\\
& \leq \frac{2}{n^{1+\gamma/2}} \int_0^{\sqn/2} t^2 \omega(t)   \int_0^{\sqn}  u^{\gamma-1}   \, du\, dt\\
& \leq \frac{M_1}{n} \int_0^{+\infty} t^2 \omega(t) dt.
\end{align*}
On the other hand, 
\begin{align*}
&\int_0^{1/2} \int_0^{x}   \int_{x-t}^{x+t} G(\sqn t) u ^{\gamma-2} du\, dt\, dx\\
& = \frac{1}{n^{\frac{1+\gamma}{2}}} \int_0^{\sqn/2} \int_0^{x}  G( t) \int_{x-t}^{x+t}  u ^{\gamma-2} du\, dt\, dx\\
& \leq \frac{1}{n^{\frac{1+\gamma}{2}}} \int_0^{+\infty} \int_0^{x}  G( t) \int_{x-t}^{x+t}  u ^{\gamma-2} du\, dt\, dx\\
&=  \frac{1}{n^{\frac{1+\gamma}{2}}} \int_0^{+\infty}  G( t) \int_t^{+\infty}  \int_{-t}^{t}  (u+x) ^{\gamma-2} du\, dx\, dt\\
&= \frac{1}{n^{\frac{1+\gamma}{2}}} \int_0^{+\infty}  G( t)  \int_{-t}^{t}  \int_t^{+\infty} (u+x) ^{\gamma-2} dx\, du\, dt\\
&\leq \frac{M_1}{n^{\frac{1+\gamma}{2}}} \int_0^{+\infty}  G( t)  \int_{-t}^{t} (t+u) ^{\gamma-1}  du\, dt\\
& \leq \frac{M_1}{n^{\frac{1+\gamma}{2}}} \int_0^{+\infty}  G( t)  t^{\gamma}\, dt.
\end{align*}
It remains to show that $\int_0^{1/2} \int_{x}^{1-x} G\left( \frac{\sqn t}{\tf(x+t)}  \right) (x+t)^{\gamma-1} dt = \mO\left( \frac{1}{n^{\frac{1+\gamma}{2}}} \right)$.
\begin{align*}
&\int_0^{1/2} \int_{x}^{1-x} G\left( \frac{\sqn t}{\tf(x+t)}  \right) (x+t)^{\gamma-1} dt \,dx\\
& \leq\int_0^{1/2} \int_{2x}^{1} G\left( \sqn( t-x) \right) t^{\gamma-1} dt \,dx\\
&=\int_0^{1} t^{\gamma-1} \int_{0}^{t/2} G\left( \sqn( t-x) \right)  dx\, dt\\
&=\int_0^{1} t^{\gamma-1} \int_{t/2}^{t} G\left( \sqn x\right)  dx\, dt\\
& =\frac{1}{n^{\frac{1+\gamma}{2}}} \int_0^{\sqn}  t^{\gamma-1}  \int_{t/2}^{t} G\left( x \right)  dx\, dt\\
& \leq \frac{1}{n^{\frac{1+\gamma}{2}}} \int_0^{+\infty}  t^{\gamma-1}  \int_{t/2}^{t} G\left( x \right)  dx \, dt,
\end{align*}
so the case $0\leq x \leq 1/2$ is complete.

\

\noindent {\bf Case 2 : $x\leq 0$}. In this case, similarly as Case 2 of the proof of Proposition \ref{prop:main}, we show that
\begin{align*}
&\int_{-\infty}^0 \bigg| \Pcro{\tth +\sqrt{\frac{\tth(1-\tth)}{n}} Z \leq x} - \Pcro{\tth \leq x}  \bigg | \; dx\\
&= \int_0^1 t^{\gamma-1} \int_{0}^{+\infty} G\left( \frac{\sqn}{\tf(t)}(t+x) \right) \,dx \,dt\\
&\leq \int_0^1 t^{\gamma-1} \int_{0}^{+\infty} G\left( \sqn (t+x) \right) \,dx \,dt\\
& = \frac{1}{n^{\frac{1+\gamma}{2}}} \int_0^{\sqn} t^{\gamma-1} \int_{t}^{+\infty} G( x ) \,dx \,dt\\
& \leq \frac{1}{n^{\frac{1+\gamma}{2}}} \int_0^{+\infty} t^{\gamma-1} \int_{t}^{+\infty} G( x ) \,dx \,dt.
\end{align*}

The bound in the cases $1/2\leq x \leq 1$ and $x\geq 1$ are proved in a similar manner.
\end{proof}

Now let us prove that $\was(\bar X_n,\theta) $ is also bounded below by a term of the order $\frac{1}{n^{\frac{1+\gamma}{2}}}$.
\begin{prop}
\label{prop:borneinf}
Under the hypothesis of Proposition \ref{prop:grando}, there exists $C>0$ such that
$$\was(\bar X_n,\theta)  \geq \frac{C}{n^{\frac{1+\gamma}{2}}}.$$
\end{prop}
\begin{proof}
From Proposition \ref{prop:equivalence}, it suffices to prove that the same type of bound holds for $\was\left(\theta + \frac{f(\theta)}{\sqrt n} Z\; , \; \theta\right)$, $Z$ being a standard normal variable independent of $\theta$. We use the dual version of the Wasserstein distance. Let $\psi$ be the $1$-Lipschitz function defined by
$$\psi(x) =\bigg| x- \frac{1}{2} \bigg|.$$
As before, let $f(x) = \sqrt{x(1-x)}$. Then
\begin{align*}
&\Ecro{\psi\left(\theta+ \frac{f(\theta)}{\sqn} Z\right)} \\
& =C_p \int_{1/2}^{3/4} \Ecro{ \bigg| t + \frac{f(t)}{\sqn} Z -\frac{1}{2}\bigg|} \left( t -\frac{1}{2}\right)^{\gamma-1}dt\\
&=C_p \int_{0}^{1/4} t  ^{\gamma-1}\Ecro{\bigg | t + \frac{f(t+1/2)}{\sqn} Z \bigg|} dt\\
& =C_p \int_{0}^{1/4} t  ^{\gamma-1}\frac{f(t+1/2)}{\sqn}\Ecro{\bigg | \frac{\sqn t}{f(t+1/2)} +  Z \bigg|} dt.
\end{align*}
However, a straightforward computation shows that for every $a \in \R$,
\begin{align*}
\Ecro{|a +Z|}= a( 1- 2 F_Z(-a) ) +2 \omega(a).
\end{align*}
Thus
\begin{align*}
&\Ecro{\psi\left(\theta+ \frac{f(\theta)}{\sqn} Z\right)} \\
 = &C_p  \int_{0}^{1/4} t  ^{\gamma-1}\left[ t \left( 1- 2 F_Z\left(-\frac{\sqn t}{f(t+1/2)}\right) \right) + \frac{2f(t+1/2)}{\sqn} \omega\left(\frac{\sqn t}{f(t+1/2)}\right) \right] dt.
 \end{align*}
 On the other hand, we have
 $$\Ecro{\psi(\theta)} = C_p  \int_{0}^{1/4} t  ^{\gamma}   dt.$$
 We obtain
 \begin{align*}
&\Ecro{\psi(\theta+ \frac{f(\theta)}{\sqn} Z)}-\Ecro{\psi(\theta)} \\
& =2C_p  \int_0^{1/4} t  ^{\gamma-1} \left[ - t  F_Z\left(-\frac{\sqn t}{f(t+1/2)}\right)  + \frac{ f(t+1/2)}{\sqn} \omega\left(\frac{\sqn t}{f(t+1/2)}\right) \right] dt\\
& =\frac{2C_p}{n^{\frac{\gamma}{2}}} \int_0^{\frac{1}{4}\sqn}  t  ^{\gamma-1} \left[ - \frac{t}{\sqn}  F_Z\left(-\frac{ t}{f\left(\frac{t}{\sqn}+1/2\right)}\right)  + \frac{ f\left(\frac{t}{\sqn}+1/2\right)}{\sqn} \omega\left(\frac{ t}{f\left(\frac{t}{\sqn}+1/2\right)}\right) \right] dt\\
& =\frac{2C_p}{n^{\frac{1+\gamma}{2}}} \int_0^{\frac{1}{4}\sqn}  t  ^{\gamma-1} \left[ - t F_Z\left(-\frac{ t}{f\left(\frac{t}{\sqn}+1/2\right)}\right)  +  f\left(\frac{t}{\sqn}+1/2\right) \omega\left(\frac{ t}{f\left(\frac{t}{\sqn}+1/2\right)}\right) \right] dt.
\end{align*}
Now, since $0\leq f(x) \leq 1$ for every $x \in [0,1]$, we have that
\begin{align*}
&\Bigg| t  ^{\gamma-1} \left[ - t F_Z\left(-\frac{ t}{f\left(\frac{t}{\sqn}+1/2\right)}\right)  +  f\left(\frac{t}{\sqn}+1/2\right) \omega\left(\frac{ t}{f\left(\frac{t}{\sqn}+1/2\right)}\right) \right] \Bigg|\\
 \leq\;&  t  ^{\gamma-1} \left[ t F_Z\left(-t\right)  +   \omega\left(t\right) \right],
\end{align*}
so that the above integral tends to $\int_0^{+\infty} t ^{\gamma-1} \left[ - t F_Z\left(-2t\right) + \frac{1}{2} \omega\left(2t\right) \right] dt$ by dominated convergence. It remains to show that this limit is not zero to achieve the proof. Integrating by parts twice, we have
$$\int_0^{+\infty}  t  ^{\gamma}   F_Z\left(-2t\right) dt = \frac{\gamma}{2(\gamma+1)} \int_0^{+\infty}t^{\gamma-1} \omega(2t) dt,$$
so that
\begin{align*}
& \int_0^{+\infty}  t  ^{\gamma-1} \left[ - t F_Z\left(-2t\right)  +  \frac{1}{2} \omega\left(2t\right) \right] dt  = \frac{1}{2(\gamma+1)} \int_0^{+\infty}t^{\gamma-1} \omega(2t) dt,
\end{align*}
which is positive.
\end{proof}

\begin{prop}
\label{prop:examples}
For every $\delta \in [1/2,1]$, there exists an infinite sequence of exchangeable 0-1 random variables $(X_k)_{k \geq 1}$ such that, if $\theta$ has the limiting distribution of $\bar X_n$,
\eq
\label{eq:bounddelta}
\was(\bar X_n, \theta) \cong \frac{1}{n^\delta}.
\qe
Conversly, if an infinite sequence of exchangeable 0-1 random variables $(X_k)_{k \geq 1}$ verifies \eqref{eq:bounddelta} for some random variable $\theta$ and some $\delta>0$, then $\delta \in [1/2,1]$.
\end{prop}
\begin{proof}
If $\delta \in (1/2,1)$ the existence of the sequence is insured by Propositions \ref{prop:grando} and \ref{prop:borneinf} (just take $\gamma = 2\delta-1$).

If $\delta = 1$, from Theorem \ref{thm:main} it suffices to choose $\theta$ with a Beta distribution.

If $\delta = 1/2$, taking $\theta$ with distribution a Dirac mass, say, at $1/2$,  it is easy to see from the very definition of the Wasserstein distance that
$$ \was \left(\theta + \sqrt{\frac{\theta(1-\theta)}{n}} Z\; , \; \theta\right) =  \frac{\E|Z|}{2\sqn},$$
and from Proposition \ref{prop:equivalence},  this implies $\was(\bar X_n,\theta) \cong 1/\sqn$.

The converse is a direct consequence of Proposition \ref{prop:boundbelow}, since if \eqref{eq:bounddelta} holds then the distribution of $\theta$ is the limiting distribution of $\bar X_n$.

\end{proof}

\section*{Acknowledgements} GP would like to thank Pietro Rigo and Antonio Lijoi for useful discussions.  GM's research is supported by a Welcome Grant from the Université de Liège. GP acknowledges support of the project F1R-MTH-PUL-15STAR / STARS at Luxembourg University.  YS  acknowledges support from the IAP Research Network P7/06 of the Belgian State (Belgian Science Policy).  

\bibliographystyle{plain} \bibliography{definetti}

\

\noindent (G. Mijoule and Y. Swan) \textsc{Département de Mathématique, Faculté des Sciences, Université de Liège, Belgium}
 
\

\noindent (G. Peccati) \textsc{Unité de Recherche en Mathématiques, Faculté des Sciences, de la Technologie et de la Communication, Université du Luxembourg, Luxembourg} 

\

\noindent\emph{E-mail address}, G. Mijoule {\tt guillaume.mijoule@gmail.com }

\noindent\emph{E-mail address}, G. Peccati {\tt giovanni.peccati@gmail.com }

\noindent\emph{E-mail address}, Y. Swan  {\tt yswan@ulg.ac.be }

\end{document}